\documentclass[preprint,12pt]{elsarticle}

\usepackage{amsbsy,amssymb,amsthm,amsmath}

\usepackage{color}
\usepackage{syntonly}
\usepackage{color}
\usepackage{graphicx}
\definecolor{violet}{rgb}{0.5,0,0.8}

\definecolor{forestgreen}{cmyk}{0.91,0,0.88,0.12}
\usepackage{t1enc}

\newcommand\R{\mathbb R}

\newcommand{\phie}{\phi_{\mathrm{e}}}

\newcommand{\xvec}{{\boldsymbol{x}}}
\newcommand{\xivec}{{\boldsymbol{\xi}}}
\newcommand{\E}{{\mathrm{e}}}
\newcommand{\I}{{\mathrm{i}}}
\let\ul=\underline

\theoremstyle{plain}
\newtheorem{thm}{Theorem}[section]
\newtheorem{lem}[thm]{Lemma}

\newtheorem{cor}[thm]{Corollary}
\newtheorem{pro}[thm]{Proposition}
\newtheorem{rem}[thm]{Remark}
\newtheorem{rmk}[thm]{Remark}

\catcode`\@=11
\@addtoreset{equation}{section}   

\journal{Nonlinear Analysis: Theory, Methods \& Applications, }

\begin{document}

\begin{frontmatter}



\title{The Dirichlet problem for $-\Delta \varphi= \mathrm{e}^{-\varphi}$ in an infinite sector.  Application to plasma equilibria.}

 \author{Olivier Goubet\corref{cor1}}
  \ead{olivier.goubet@u-picardie.fr}
 \cortext[cor1]{LAMFA CNRS UMR 7352, Universit\'e de Picardie Jules Verne, 33, rue Saint-Leu, 80 039 Amiens, France.}

 \author{Simon Labrunie\corref{cor2}}
  \ead{simon.labrunie@univ-lorraine.fr}
 \cortext[cor2]{Université de Lorraine, Institut Elie Cartan de Lorraine, UMR 7502,  54506 Vand{\oe}uvre-lès-Nancy, France.\\
CNRS, Institut Elie Cartan de Lorraine, UMR 7502,  54506 Vand{\oe}uvre-lès-Nancy, France.}

\begin{abstract}
We consider here a nonlinear elliptic equation in an unbounded sectorial domain of the plane.
{We prove the existence of a minimal solution to this equation and study its properties.}
We infer from this analysis some asymptotics for the stationary solution of an equation arising
in plasma physics.
\end{abstract}

\end{frontmatter}



\section{Setting of the problem}
\label{sec:intro}

Solving elliptic PDE in unbounded domains
of $\mathbb{R}^n$ such as half-spaces occurs naturally
when using some blow-up argument to analyze the properties
of a particular solution of a PDE in a bounded domain.
We refer for instance to~\cite{angenent} where the
analysis of the properties of the solution $u_\varepsilon$ to
\begin{equation}\label{pde1}
\begin{split}
\varepsilon \Delta u_\varepsilon + f(u_\varepsilon)=0,\\
u_\varepsilon=0 \; \mbox{on the boundary}
\end{split}
\end{equation}
in a neighborhood of a point of the boundary,
leads naturally to the study of an elliptic PDE in the half-space.

The nonlinear elliptic PDE
\begin{equation}\label{pde2}
\Delta u + f(u)=0,
\end{equation}
has been widely studied in domains as half-spaces or cylindrical unbounded domains.
We refer here to the articles \cite{BCN93, BCN96, BCN97, BCN98} which have been instrumental
for any later results concerning the symmetry and monotonicity properties of solutions.
In the literature, there are various results concerning the properties
of {\it bounded} solutions to these equations, mainly using consequences of
the maximum principle as the moving plane method or the sliding method
(see~\cite{FS,EL83},\dots) We also refer to~\cite{EH11}
where the properties of solutions in a quarter-space have been studied
using tools from infinite-dimensional dynamical systems.

\medbreak

Let us now describe the equation we are interested in. Consider $\Omega$ a sectorial domain of $\mathbb{R}^2$ defined in polar coordinates as:
$$\Omega = \{\xvec(r,\theta)  \in \mathbb{R}^2; |\theta|< \theta_0\leq \pi\}.$$
We shall sometimes denote this set as $\Omega[\theta_0]$ when we need to specify the opening.
We are interested in the \emph{non-negative} solutions to the problem:
\begin{equation}\label{pde}
\begin{split}
-\Delta \varphi= \E^{-\varphi} \quad\text{in } \Omega, \\
\varphi=0 \quad\text{on } \partial\Omega.
\end{split}
\end{equation}
Our motivation here comes from the article \cite{KaLP13}, where the authors study stationary solutions to {the} Vlasov--Poisson system in a polygon
{and link them to those of a non-linear elliptic equation.}
The singular limit of the {latter} while some {scaling parameter} converges towards $0$ {leads to~\eqref{pde}}.
{Our aim in analyzing this equation is} to provide more insight on the solutions to the original Vlasov--Poisson equation.

\medbreak

More specifically, we shall look for two types of solutions {to~\eqref{pde}}. \emph{Local variational solutions} satisfy
$\varphi \in H^1(\mathcal{O})$ for any bounded open set $\mathcal{O} \subset \Omega$.
{By the Trudinger inequality, this implies $\int_{\mathcal{O}} \exp(\varphi^2) < +\infty$, and thus $\E^{-\varphi} \in L^2(\Omega)$. In this}
case the Dirichlet condition holds in the sense of the usual trace theory:
\begin{equation}\label{pdevar}
\int_\Omega \nabla\varphi\cdot\nabla v= \int_\Omega \E^{-\varphi} v,\quad \forall v\in H^1_0(\Omega) \text{ with {bounded} support.}
\end{equation}
\emph{Very weak solutions} are such that { $\varphi \in L^2(\mathcal{O})$ and $\E^{-\varphi} \in H^{-1}(\mathcal{O})$:}
\begin{equation}\label{pdeweak}
-\int_\Omega \varphi\, \Delta v= \int_\Omega \E^{-\varphi}\, v,\quad \quad \forall v\in H^2\cap H^1_0(\Omega) \text{ with {bounded} support.}
\end{equation}
The trace is defined in a very weak sense on any bounded subset of each side of~$\partial\Omega$, {by an immediate generalization of~\cite{Gris92}. Anyway, as we are interested in non-negative solutions, there automatically holds $\E^{-\varphi} \in L^\infty(\mathcal{O}) \subset H^{-1}(\mathcal{O})$.}

\medbreak

\noindent For both types of solutions, there obviously holds:
\begin{lem}
\label{pro:isom}
Let $\varphi$ be a solution to~\eqref{pde} on the sector~$\Omega$. For any isometry $\mathcal{T}$ of~$\R^2$ (translation, rotation, reflection), the function $ \varphi(\mathcal{T}\xvec)$ is a solution on the sector~$ \mathcal{T}^{-1}(\Omega)$.
\end{lem}

\medbreak

Since $0$ is a subsolution to the problem~\eqref{pde}, the existence of a solution is equivalent to that of a non-negative supersolution to the problem.
We will develop this in the sequel. Furthermore, any non-negative solution to~\eqref{pde} in~$\Omega[\theta_0]$ is a supersolution in a smaller sector $\Omega[\theta_1],\ \theta_1 < \theta_0$. Therefore, the existence of a solution in the split plane~$\Omega[\pi]$ implies the solvability of~\eqref{pde} in any sector.
By symmetry, it is enough to solve the mixed Dirichlet--Neumann problem in the upper half-plane
\begin{equation}
\begin{split}
-\Delta \varphi_* = \E^{-\varphi_*} \quad\text{in } \R^2_+ = [0 < \theta < \pi], \\
\varphi_* = 0 \text{ on } [\theta = \pi] ,\quad \partial_n \varphi_* = 0 \text{ on } [\theta = 0].
\end{split}
\label{pde:mixed}
\end{equation}
Glueing this $\varphi_*$ to its even reflection with respect to the axis $[\theta=0]$ yields a solution to~\eqref{pde} on~$\Omega[\pi]$.

\begin{rem}
It is worth pointing out that there exists no {(non-negative, very weak)} solution to~\eqref{pde} in the whole plane.
\end{rem}

The article is written as follows. In a second section
we construct a supersolution to the equation in the split plane. For this purpose
we use a constructive method which relies on complex analysis.
In a third section  we discuss some properties, {such} as monotonicity, symmetry or regularity, of the \emph{minimal} positive solution; this minimal
solution is relevant for the Physics of the original problem.
In a fourth section we prove the non-uniqueness of solutions and list some of their properties.
Eventually, we discuss the application to the stationary Vlasov--Poisson system
in a last section.
In this last section, we also {show the link between this asymptotic and the} boundary blow-up solutions for $\Delta u = \E^u$, {see}~\cite{BM92, BM95, MV03, LM94, DDGR07}. Boundary blow-up (or large) solutions {were} introduced in the seminal articles~\cite{keller, osserman}.

\section{Construction of a solution in the split plane}
We begin with a construction inherited from complex analysis.
It is worth pointing out that {this construction} method works
for any sectorial domain.

\begin{pro}
There exists a solution $\varphi_*$ to equation \eqref{pde:mixed} in the split plane.
\end{pro}

Let  $z_2 = \Phi(z_1)$ or $z_1 = \Psi(z_2)$ be a conformal mapping between the complex $z_1$ and $z_2$-planes, and let $D_1,\ D_2$ be two domains conformally mapped to one another. Suppose we are given two functions $w_1,\ w_2$ on $D_1$ and $D_2$ respectively, which are transformed into each other by the formulas:
\begingroup
\begin{eqnarray}
w_2(z_2) &=& \log |\Psi'(z_2)| + w_1( \Psi(z_2) ) ;
\label{w2w1}\\
w_1(z_1) &=& \log |\Phi'(z_1)| + w_2( \Phi(z_1) ) .
\label{w1w2}
\end{eqnarray}
\endgroup
Using $\Delta = 4\,\partial_z\partial_{\bar z}$ and the fact that the logarithm of the modulus of an analytic function is harmonic, one easily checks the following lemma.

\begin{lem}
\label{lem:w1w2}
Let $\Delta_i,\ i=1,\ 2$ be the Laplacian in the $z_i$-plane, and let  $w_1,\ w_2$ be related by~\eqref{w2w1} or~\eqref{w1w2}.
Then, $w_1$ satisfies $\Delta_1 w_1 = 4\E^{2w_1}$ in~$D_1$ if, and only if, $w_2$ satisfies $\Delta_2 w_2 = 4\E^{2w_2}$ in~$D_2$.
\end{lem}

\medbreak

In this section, we denote $(x_i,y_i)$ and $(r_i,\theta_i)$ the Cartesian and polar coordinates in the $z_i$-plane. We {choose} $D_1$ as the upper half-plane $[0 < \theta_1 < \pi] = [{y_1} > 0]$, and we introduce the mixed Dirichlet--Neumann problem
\begin{equation}
\Delta_1 w_1 = 4\E^{2w_1} \text{ in } D_1,\quad w_1 = 0 \text{ on } \Gamma_1^D,\quad \partial_{n_1} w_1 = 0 \text{ on } \Gamma_1^N,
\label{pbw1m:max}
\end{equation}
where  $\Gamma_1^D = [\theta_1 = \pi] = \{(x_1,0) : x_1<0\}$, and $\Gamma_1^N = [\theta_1=0] = \{(x_1,0) : x_1>0\}$. This is obviously related to~\eqref{pde:mixed} (set $\varphi_*(\xvec) = -2 w_1(\xvec/\sqrt{8})$).

\medbreak

Consider now the conformal mapping
$$\Phi(z_1)=\frac{1}{z_1+\I},\quad \Psi(z_2)=\frac{1}{z_2}-\I.$$
The half-plane is mapped by $\Phi$ onto the disk $D_2$ centered at $\frac{1}{2\I}$ and of radius $\frac 1 2$; the negative and positive real half-axes $\Gamma_1^D$ and~$\Gamma_1^N$ are respectively mapped to the left and right half-circles $\Gamma_2^D$ and~$\Gamma_2^N$ of this disk.

\medbreak

By Lemma~\ref{lem:w1w2}, the function $w_2$ defined by~\eqref{w2w1} satisfies $\Delta_2 w_2 = 4\E^{2w_2}$ in~$D_2$.
What about the boundary conditions? On the Dirichlet half-circle~$\Gamma_2^D$, one has
\begin{equation}\label{halfdirichlet}
w_2(z_2)=\log |\Psi'(z_2)|=-\log|z_2|^2\geq 0;
\end{equation}
it is worth observing that this function is non-negative and singular at $z_2=0$ (corresponding to $z_1=\infty$) only.
On the Neumann half-circle~$\Gamma_2^N$, we compute as follows. In polar coordinates, we have:
\begin{equation}\label{calcul1}
w_2(r_2\,\E^{\I\theta_2})=-2\log r_2 + w_1\left( \frac{\cos \theta_2}{r_2} - \I\left(1+\frac{\sin\theta_2}{r_2}\right) \right).
\end{equation}
Parametrizing the boundary as $z_2=\frac12\left(\E^{\I \vartheta}-\I\right)$, we observe that:
\begin{eqnarray*}
&&r_2^2=\frac{1-\sin\vartheta}{2},\quad r_2\cos\theta_2=\frac12\,\cos\vartheta,\quad r_2\sin\theta_2=\frac{-1+\sin\vartheta}{2} \,;
\\
&&2\theta_2= \vartheta- \frac \pi 2,\quad (\cos \vartheta, \sin \vartheta)=(-\sin 2\theta_2,\cos 2\theta_2).
\end{eqnarray*}
Recalling that $\partial_{y_1}w_1= - \partial_{n_1}w_1=0$ on~$\Gamma_1^N$, we deduce:
\begin{eqnarray*}
\label{calcul2}
\partial_{x_2}w_2 &=& -\frac{2\cos \theta_2}{r_2}-\frac{\cos 2\theta_2}{r_2^2}\partial_{x_1}w_1 \quad {\text{on } \Gamma_2^N}, \\
\label{calcul3}
\partial_{y_2}w_2 &=& -\frac{2\sin \theta_2}{r_2}-\frac{\sin 2\theta_2}{r_2^2}\partial_{x_1}w_1 \quad {\text{on } \Gamma_2^N}.
\end{eqnarray*}
Therefore
\begin{equation}\label{calcul4}
\partial_{n_2}w_2=\frac{2\sin \theta_2}{r_2}=-2 \quad \text{on } \Gamma_2^N.
\end{equation}

\medbreak

Next, we proceed to a truncation on the boundary condition, setting $g_2^k=\min(k, -\log|z_2|^2)$, i.e., we solve the mixed Dirichlet--Neumann problem
\begin{equation}
\Delta_2 w_2^k = 4\E^{2w_2^k} \text{ in } D_2,\quad {w_2^k} = g_2^k \text{ on } \Gamma_2^D,\quad \partial_{n_2} {w_2^k} = -2 \text{ on } \Gamma_2^N.
\label{pbw2km:max}
\end{equation}
The unique variational solution minimizes the strictly convex l.s.c. functional
$$\int_{D_2} \left\{ \frac12\,| \nabla w |^2 + 2\,\E^{2w} \right\} + \int_{\Gamma_2^N} 2\,w$$
on the affine space $V(g_2^k) := \{ w \in H^1(D_2) : w = g_2^k \text{ on } \Gamma_2^D \}$.
We remark that $w_2^k \in C(\overline{D_2})$:
{as argued in Section~\ref{sec:intro}, $\Delta_2 w_2^k = \E^{2w_2^k} \in L^2(D_2)$ by the Trudinger inequality};
as the boundary data are smooth enough (actually, {$g_2^k \in H^{3/2-\epsilon}(\Gamma_2^D)$}), one has $w_2^k \in H^{3/2-\epsilon}(D_2) \subset C(\overline{D_2})$; the regularity is limited by the change of boundary conditions~\cite{Gris92}.

\medbreak

\noindent We now state that $u_2(z_2)=-\log|z_2|^2$ is a supersolution to~\eqref{pbw2km:max}.
\begin{lem}
For any $z_2$ in $D_2$, $w_2^k(z_2)\leq -\log|z_2|^2$.
\end{lem}
\begin{proof}
Computing as above, we see that $u_2$ solves the problem
$$ \Delta_2 u_2 = 0 \text{ in } D_2,\quad u_2 = -\log r_2^2 \text{ on } \Gamma_2^D,\quad \partial_{n_2} u_2 = -2 \text{ on } \Gamma_2^N. $$
It is not a variational solution, but it is smooth where it is bounded. Moreover, $u_2\ge0$ in~$D_2$. As a consequence, there holds:
$$ \int_{D_2} \nabla u_2 \cdot \nabla v + \int_{\Gamma_2^N} 2\,v = 0,$$
for all $v\in H^1(D_2)$ such that $v = 0$ on~$\Gamma_2^D$ and $u_2$ is bounded above on~$\mathrm{supp}\, v$.
On the other hand, the variational formulation of~\eqref{pbw2km:max} reads
\begin{equation}
\label{varfor}
\int_{D_2} \nabla w_2^k \cdot \nabla v + \int_{D_2} 4\E^{2w_2^k}\, v= - 2\int_{\Gamma^N_2} v.
\end{equation}
Therefore
$$ \int_{D_2} \left\{ \nabla (w_2^k - u_2) \cdot \nabla v + 4\E^{2w_2^k}\,v \right\}= 0,$$
for any admissible~$v$. The function $v = (w_2^k - u_2)_+$ is admissible: it vanishes both on~$\Gamma_2^D$ and where $u_2 \ge \max w_2^k$, hence $u_2$ is bounded above on the support of~$v$. One finds as usual $(w_2^k - u_2)_+ = 0$, i.e., $w_2^k \le u_2$.
\end{proof}

\medbreak

Obviously, $w_2^K$~also is a supersolution to~\eqref{pbw2km:max} for $K>k$; therefore,
$$w_2^0\le w_2^k \le  w_2^K \le u_2,\quad {\text{ for } 0 \le k \le K}.$$
We deduce that, for all $z_2\in \overline{D_2} \setminus \{0\}$, $w_2^k(z_2)$ converges to a limit $w_2(z_2) \le u_2(z_2)$. By the monotone convergence theorem, $w_2^k \to w_2$ and $\E^{2w_2^k} \to \E^{2w_2}$ in $L^p(\mathcal{K})$, for any compact $\mathcal{K} \subset \overline{D_2} \setminus \{0\}$ and $p < \infty$. This implies $\Delta_2 w_2 = 4\E^{2w_2}$ in the sense of distributions in~$D_2$. Furthermore, one can consider boundary conditions in a very weak sense~\cite{Gris92}.
Summing up, $w_2$ solves the problem:
$$\Delta_2 w_2 = 4\E^{2w_2} \text{ in } D_2,\quad w_2 = -\log |z_2|^2 \text{ on } \Gamma_2^D,\quad \partial_{n_2} w_2 = -2 \text{ on } \Gamma_2^N.$$
As $w_2^0$ is bounded on~$\overline{D_2}$, the limit is bounded below: $w_2 \ge m$ on $\overline{D_2}$.
Using the conformal mapping~$\Phi$, the function $w_1$ defined by~\eqref{w1w2} is a solution to~\eqref{pbw1m:max}
which satisfies the bounds
\begin{equation}
m + \log |\Phi'(z_1)|= m - \log(x_1^2+(1+ y_1)^2)\leq w_1(z_1)\leq 0 \quad\text{on } \overline{D_1} .
\label{inequation}
\end{equation}

\section{The minimal solution}
{F}rom the previous section, we have constructed a solution $\varphi_*$ to~\eqref{pde:mixed}, which satisfies the bounds (cf.~\eqref{inequation}):
\begin{equation}
0 \le \varphi_*(\xvec) \le 2\,\log(1 + {\tfrac{\sqrt{2}}2}\,r\,\sin\theta + \tfrac18\,r^2 ) - 2\,m
\label{bnd:phi*}
\end{equation}
Thus, $\varphi_*$, extended by reflection to~$\Omega[\pi]$,
{ is bounded on any bounded subset of~$\Omega[\pi]$.}
We will provide a better version of this upper bound in the sequel.

\smallbreak

Also, notice that the conditions, valid for any bounded subset $\mathcal{O} \subset \Omega[\pi]$
$$\varphi_* \in L^\infty(\mathcal{O}), \quad \Delta \varphi_* = \E^{-\varphi_*}\in L^\infty(\mathcal{O}),\quad \varphi_* = 0 \text{ on } \partial\Omega[\pi]$$
actually imply $\varphi_* \in H^1(\mathcal{O})$ (see~\cite{Gris92} or Proposition~\ref{local-rgty} below);{therefore $\varphi_*$ is a local variational solution to~\eqref{pde} in~$\Omega[\pi]$.}

\subsection{A limiting process}\label{limitingproblem}
Consider the truncated domain
$$\Omega_R= \{\xvec(r,\theta) \in \mathbb{R}^2; |\theta|< \theta_0\leq \pi,\ 0 < r <R \}.$$
\noindent
The proof of the following result is standard.
\begin{lem}\label{truncatedproblem}
There exists a unique {variational} solution $u_R$ to the problem
\begin{equation}\label{pdeR}
\begin{split}
-\Delta u_R= \E^{-u_R} \quad\text{in } \Omega_R, \\
u_R=0 \quad\text{on } \partial\Omega_R.
\end{split}
\end{equation}
This solution is positive and is symmetric with respect to $\theta\mapsto -\theta$.
\end{lem}

Let us observe that, for $R'\ge R$, both $u_{R'}$ and $\varphi_*$ are supersolutions to \eqref{pdeR}.
Therefore, for any $\xvec$ in $\Omega_R$
$$ u_R(\xvec)\leq u_{R'}(\xvec) \leq \varphi_*(\xvec).$$
Passing to the limit while $R$ goes to the infinity, setting $u(\xvec)=\sup_R u_R(\xvec)$
we construct this way a solution $u$ to the original problem.
This solution is symmetric with respect to $\theta=0$.
Moreover,
$u$ is the {\it minimal} solution, in the sense that any (non-negative) {local variational} solution $\varphi$ to \eqref{pde} satisfies $u\leq \varphi$;
(\footnote{ This is also true for very weak solutions, as a consequence of Proposition~\ref{local-rgty-gene} below.})
therefore, $u$ is unique. Similarly, any non-negative supersolution $\overline{\varphi}$ {of local variational regularity} satisfies $u\leq \overline{\varphi}$.
{F}rom~\eqref{bnd:phi*}, we deduce that there exists $C\ge0$ such that
\begin{equation}
u(\xvec) \le C + 4 \log (1+|\xvec|),\quad \forall \xvec \in \Omega.
\label{bnd:u}
\end{equation}


\begin{rem}\label{unbounded}
We can compute the minimal solution {in the half-plane~$\Omega[\pi/2]$, which reduces to a}  1D problem on the half-line, that is $2\log(1+\frac{x_1}{\sqrt 2})$. Since the minimal
solution $u_{[\theta_0]}$ in $\Omega[\theta_0]$ is a supersolution {if $\theta_0 \ge \pi/2$}, then $u_{[\theta_0]}$ cannot be bounded by above in that case.
\end{rem}

\subsection{Regularity}\label{ssc:rgty}

We introduce the following function spaces on a domain~$\mathcal{O}$:
\begin{eqnarray*}
\Phi_p(\mathcal{O}) &:=& \left\{ w\in H^1_0(\mathcal{O}) : \Delta w \in L^p(\mathcal{O}) \right\}, \\
N(\mathcal{O}) &:=& \text{ orthogonal of } \Delta\left[ H^2(\mathcal{O}) \cap H^1_0(\mathcal{O}) \right] \text{ within } L^2(\mathcal{O}) \\
 &=& \left\{ p\in L^2(\mathcal{O}) : \Delta p = 0 \text{ in } \mathcal{O} \text{ and } p=0 \text{ on each side of } \partial\mathcal{O} \right\}.
\end{eqnarray*}
In the first two lines, $\mathcal{O}$ is a bounded Lipschitz domain. The second characterisation of~$N(\mathcal{O})$, proved in~\cite{Gris92} for polygons, can be extended to curvilinear polygons%
\footnote{A \emph{curvilinear polygon} is an open set $\mathcal{O}\subset\R^2$ such that for any $\xvec_0 \in \overline{\mathcal{O}}$ and $\eta$ sufficiently small,  $\mathcal{O} \cap B(\xvec_0,\eta)$ is $C^2$-diffeomorphic, either to~$\R^2$, or to a sector~$\Omega[\theta_{\xvec_0}]$, with $0 < \theta_{\xvec_0} < \pi$. This definition includes both smooth domains and straight polygons, but excludes cusps and cracks.}
such as~$\Omega_R$, {whose boundary is composed of smooth sides that meet at corners}.
To express the regularity of solutions and give asymptotic expansions, we shall generally make use of the parameter:
$$\alpha = \pi/(2\theta_0).$$
We introduce the well-known primal and dual harmonic singularities in~$\Omega$:
\begin{equation}
S(r,\theta) := r^{\alpha}\, \cos(\alpha\theta),\quad
S^*(r,\theta) := \tfrac1{\pi}\, r^{-\alpha}\, \cos(\alpha\theta),
\end{equation}
From~\cite{Gris92}, we know the following facts.
\begin{pro}
\label{pro:grisvard}
Let $\chi$ be a fixed cutoff function equal to~$1$ for $|\xvec| \le 1$ and $0$~for $|\xvec| \ge 2$; for any  $B>0$ we write $\chi_B(r) := \chi(r/B)$.
\begin{itemize}
\item If the angle at the tip of the sector is salient or flat, then:
$$N(\Omega_R) = \{0\},\quad \Phi_2(\Omega_R) = H^2\cap H^1_0(\Omega_R),\quad \Phi_p(\Omega_R)\subset W^{2,p}(\Omega_R),$$
for some $p>2$.
\item If the sector is reentrant, the space $N(\Omega_R)$ is one-dimensional, spanned by
$$P_s^R(r,\theta) =  \chi_{R/2}(r)\,S^*(r,\theta) + \tilde P^R , \quad\text{with}\quad \tilde P^R \in H^1_0(\Omega_R).$$
The spaces $\Phi_p(\Omega_R)$ for $p\ge2$ are embedded into $u\in H^s(\Omega_R) \subset C(\overline{\Omega_R})$, for all $s < 1 + \alpha$; any $w\in\Phi_2$ admits the regular-singular decomposition:
$$w = \lambda[w]\, \chi_{R/2}\, S + \tilde w,\quad \text{with}\quad \tilde w\in H^2\cap H^1_0(\Omega_R),$$
where $\lambda[w]$ is called the \emph{singularity coefficient} of~$w$.
\end{itemize}
\end{pro}
\begin{rmk}
\label{rmk:psr}

Actually, $P_s^R$ can be computed exactly in~$\Omega_R$:
\begin{equation}
\label{psr:Omega_R}
P_s^R(r,\theta) = \frac1{\pi} \, \left(r^{-\alpha}- \left(\frac{r}{R^2}\right)^\alpha \right)\, \cos(\alpha\theta).
\end{equation}
\end{rmk}

\medbreak

As a first step, we prove that the minimal solution is a local variational solution.

\begin{pro}
\label{local-rgty}
Let $R$ be an arbitrary positive number. The minimal solution $u$ has the following regularity in~$\Omega_R$.
\begin{enumerate}
\item If the sector is salient of flat ($\theta_0 \le \pi/2$), $u\in C^1(\overline{\Omega_R})$.
\item It the sector is reentrant ($\theta_0 > \pi/2$), $u\in H^s(\Omega_R) \subset C(\overline{\Omega_R})$, for all $s < 1 + \alpha$; and it admits the expansion $u(r,\theta) = \Lambda\, S(r,\theta) + \tilde u(r,\theta)$, with $\tilde u \in H^2(\Omega_R)$ and $\Lambda>0$.
\end{enumerate}
\end{pro}
\begin{proof}
It follows from~\eqref{bnd:u} that $u$ is bounded on~$\Omega_R$ for any~$R$. Fix $B>0$, and introduce $v \in H^1_0(\Omega_{2B})$ which solves
\begin{equation}
\Delta v = f := \Delta(\chi_B\,u) = -\chi_B\,\E^{-u} + 2\, \nabla\chi_B \cdot \nabla u + (\Delta\chi_B)u \in H^{-1}(\Omega_{2B}),
\label{eq:chi.u}
\end{equation}
and set $P := v - \chi_B\,u$. There holds $P \in L^2(\Omega_{2B}), \ \Delta P = 0$ and $P=0$ on each side of~$\partial\Omega_{2B}$; so $P \in N(\Omega_{2B})$. Therefore, $P=0$ and $\chi_B\, u = v \in H^1_0(\Omega_{2B})$ in the salient case. In a reentrant sector, we have $P = A\,P_s^R$.
{ On the other hand, a Sobolev embedding gives $v \in L^p(\Omega_{2B})$ for all $p<+\infty$; as $\chi_B\, u \in L^\infty(\Omega_{2B})$ by~\eqref{bnd:u}, one has $P \in L^p(\Omega_{2B})$ for all $p<+\infty$, which is only possible if $A=0$, i.e., $\chi_B\, u = v \in H^1_0(\Omega_{2B})$ again.}

\smallbreak

Thus, we have in all cases $u\in H^1(\Omega_B)$; computing as in~\eqref{eq:chi.u}, we find $\Delta(\chi_{B/2}\,u) \in L^2(\Omega_B)$, i.e., $ \chi_{B/2}\,u \in \Phi_2(\Omega_B)$.
Invoking Proposition~\ref{pro:grisvard} again, this leads to:
\begin{itemize}
\item $ \chi_{B/2}\,u \in H^2(\Omega_B) \subset W^{1,p}(\Omega_B)$ for all $p<+\infty$, in the salient case. Bootstrapping again, we deduce $\Delta(\chi_{B/2}\,u) \in L^p(\Omega_B)$, and $ \chi_{B/2}\,u \in \Phi_p(\Omega_B) \subset W^{2,p}(\Omega_B)$ for some $p>2$. In other words, $u\in W^{2,p}(\Omega_R) \subset C^1(\overline{\Omega_R})$ as $B$ is arbitrary.
\item  in the reentrant case, we get the expansion
$$ (\chi_{B/2}\,u)(r,\theta) = \Lambda\, \chi_{B/2}(r)\, S(r,\theta) + \tilde v(r,\theta), \quad\text{with}\quad \tilde v \in H^2 \cap H^1_0(\Omega_B),$$
hence the decomposition of~$u$ as $B$ is arbitrary.
\end{itemize}

\medbreak

To prove $\Lambda>0$, we return to the solution~$u_R$ on~$\Omega_R$. Let $\Lambda_R$ be its singularity coefficient; it controls the dominant behaviour of~$u_R$ near the corner, $u_R \sim \Lambda_R\, S(r,\theta)$ as $r\to0$, uniformly in~$\theta$~\cite{KaLP13}. It is given by the formula~\cite{Gris92}:
$$ \Lambda_R = \int_{\Omega_R} (-\Delta u_R)\,P_s^{R} = \int_{\Omega_R} \E^{-u_R}\,P_s^{R}.$$
{{F}rom~\eqref{psr:Omega_R}, we see that $P_s^R > 0$ in~$\Omega_R$.}%
(\footnote{ This property is probably true for any curvilinear polygon. Anyway, it is easily seen that the dual singular function is strictly positive on some subdomain, which is enough.})
One infers that $\Lambda_R > 0$. Yet, we have seen that $u\ge u_R$, thus $\Lambda \ge \Lambda_R>0$.
\end{proof}

\medbreak

\begin{rem}
\label{rem:LambdaR}
Actually, there holds $\Lambda = \lim_{R\to+\infty} \Lambda_R$. In the fixed domain~$\Omega_{2B}$, the function $v_R := \chi_B\,u_R \in H^1_0(\Omega_{2B})$ solves
$$ \Delta v_R = f_R := -\chi_B\,\E^{-u_R} + 2\, \nabla\chi_B \cdot \nabla u_R + (\Delta\chi_B)u_R \in L^2(\Omega_{2B}).$$
By the monotone convergence theorem, we know that $u_R\to u$ and $\E^{-u_R}\to \E^{-u}$ in~$L^2(\Omega_{2B})$. Hence $f_R \to f$ in~$H^{-1}(\Omega_{2B})$ and $v_R \to v$ in~$H^1_0(\Omega_{2B})$. In particular, $u_R \to u$ in~$H^1(\Omega_B)$, and (as in Proposition~\ref{local-rgty}), $\chi_{B/2}\,u_R \to \chi_{B/2}\,u \in \Phi_2(\Omega_B)$. As the singularity coefficient is a continuous linear form on~$\Phi_2$, one infers $\Lambda_R \to \Lambda$.
\end{rem}

\medbreak

\subsection{Monotonicity}

\noindent We can now prove a monotonicity property.
 Standard results in the literature assert that {\it bounded}
solutions to this type of PDE have some monotonicity properties.
Our minimal solution is not bounded, see Remark \ref{unbounded}.

\begin{pro}\label{increasing}
Let $\boldsymbol{u}[\theta]$ be the vector $(\cos\theta,\sin\theta)$.
For any $\theta$ such that $|\theta|\leq \theta_0$, $t\ge0$ and $\xvec\in\Omega$, there holds $u(\xvec) \le u(\xvec + t\,\boldsymbol{u}[\theta])$.
\end{pro}
\begin{rem}
This implies that $\xvec \cdot \nabla u(\xvec
)\geq 0$ and $\frac{\partial u}{\partial x_1}\geq 0$.
\end{rem}
\begin{proof}
We use Proposition~\ref{pro:isom} with $\mathcal{T}$ being the translation $\mathcal{T}\xvec = \xvec + t\,\boldsymbol{u}[\theta]$.
The function $u(\mathcal{T}\xvec)$ is a super-solution in $\Omega$, as $\Omega\subset \mathcal{T}^{-1}(\Omega)$,
so $u(\mathcal{T}\xvec)\geq u(\xvec)$ and the result follows promptly.
\end{proof}

\begin{lem}\label{derivative}
The minimal solution satisfies
\begin{equation}\label{upperderivative}
r\, \partial_r u(r,\theta) = \xvec\cdot\nabla u(\xvec)\leq 2.
\end{equation}
\end{lem}
\begin{proof}
We observe that for any $\lambda\in (0,1)$, $u(\lambda\xvec)-2\log \lambda$ is a supersolution to the equation \eqref{pde}. Thus:
\begin{equation}\label{21}
u(\xvec)\leq u(\lambda \xvec)-2\log \lambda.
\end{equation}
Dividing the equation \eqref{21} by $1-\lambda$ and letting $\lambda\rightarrow 1$ yields the result.
\end{proof}
\begin{rem}\label{upperbound}
Setting $\lambda = |\xvec|^{-1}$ in~\eqref{21}, we find a slightly improved version of~\eqref{bnd:u}:
\begin{equation}\label{upper}
u(\xvec)\leq \sup_{|\boldsymbol{y}| = 1} u(\boldsymbol{y}) + 2\log |\xvec| = \sup_{|\boldsymbol{y}|\leq 1} u(\boldsymbol{y}) + 2\log |\xvec|,\quad
\forall |\xvec| \ge1.
\end{equation}
The equality of the two upper bounds follows from Proposition~\ref{increasing}.
\end{rem}

\subsection{Symmetry}\label{ssc:sym}
\begin{pro}
Consider $u_R$ the solution of \eqref{pdeR}. Then $\theta \mapsto u_R(r,\theta)$
achieves its maximum at $\theta=0$.
\end{pro}
\begin{cor}
The minimal solution $u$ of \eqref{pde} satisfies that $\theta \mapsto u(r,\theta)$
achieves its maximum at $\theta=0$.
\end{cor}
\begin{proof}
Consider the half-domain $\Omega_R^+ := [ 0<\theta<\theta_0 \text{ and } r<R]$.
The function $v:=\partial_\theta u_R$ is solution to
\begin{equation}
-\Delta v+\E^{-u_R}v=0,
\end{equation}
and $v\leq 0$ on the boundary. Actually, $v=0$ on the lower ray $[\theta=0]$ and on the arc of circle $[r=R]$, while $v\leq 0$ on the upper ray $[\theta=\theta_0]$. By the maximum principle,%
\footnote{To check that $v\in H^1(\Omega_R^+)$: if the sector $\Omega$ is salient or flat ($\theta_0 \le \frac\pi2$), write: {$v = \xvec^\perp\cdot\nabla u_R$, with $\xvec^\perp := (-x_2,x_1)$.
As $u_R \in H^2(\Omega_R)$ in this case, one deduces $v \in H^1(\Omega_R)$.}
If the sector is reentrant ($\theta_0 > \frac\pi2$), write $u_R = \tilde u_R + \Lambda_R\,S$, with $\tilde u_R \in H^2(\Omega_R)$. Then:
$$ v = \partial_\theta \tilde u_R + \Lambda_R\,\partial_\theta S = {\xvec^\perp\cdot\nabla\tilde u_R} + \Lambda_R\,(-\alpha\,r^\alpha\,\sin(\alpha\theta))$$
The second term belongs to~$H^s(\Omega_R)$ for $s<1+\alpha$, and the first is treated as above.}
$v<0$ in~$\Omega_R^+$. The results follows promptly.
\end{proof}

\begin{rem}
We reckon that the map $\theta \mapsto u(r,\theta)$ is concave, but we do
not have a proof of this fact.
\end{rem}

\section{General properties of solutions}

\subsection{Non-uniqueness}
For any $\mu = (\mu_-,\mu_+) {\in (\R^+)^2}$, the function $H_\mu := \mu_-\,S^* + \mu_+\,S$ is non-negative and harmonic in~$\Omega$, and vanishes on the boundary: it is a subsolution to~\eqref{pde}. We shall construct a solution in the form $\varphi_\mu = H_\mu + v^\mu$. If such a solution exists, then $v^\mu$ solves
\begin{equation}
-\Delta v^\mu = \E^{-H_\mu} \, \E^{-v^\mu} \quad\text{in } \Omega,\quad
v^\mu = 0 \quad\text{on } \partial\Omega.
\label{vmu}
\end{equation}
The corresponding problem in the truncated domain
\begin{equation*}
-\Delta v^\mu_R = \E^{-H_\mu} \, \E^{-v^\mu_R} \quad\text{in } \Omega_R,\quad
v^\mu_R = 0 \quad\text{on } \partial\Omega_R,
\end{equation*}
is well-posed, and $v^\mu_{R'}$ is a supersolution when $R'>R$. The minimal solution $u$ to~\eqref{pde} also is a supersolution, so $v^\mu_R \le v^\mu_{R'} \le u$, and passing to the limit we obtain a solution to~\eqref{vmu}. Thus we have constructed a solution $\varphi_\mu = H_\mu + v^\mu$ to~\eqref{pde}, which is bounded as
$$ \mu_-\,S^* + \mu_+\,S \le \varphi_\mu \le \mu_-\,S^* + \mu_+\,S + u. $$
Because of~\eqref{bnd:u}, there holds
\begin{eqnarray*}
\text{if } \mu_- >0,&&
\varphi_\mu(r,\theta) \sim\tfrac1\pi\, \mu_-\, r^{-\alpha}\,\cos(\alpha\theta) \quad\text{as } r\to0,\quad \theta \ne\pm\theta_0 \text{ fixed} ; \\
\text{if } \mu_+ >0,&&
\varphi_\mu(r,\theta) \sim  \mu_+\, r^{\alpha}\,\cos(\alpha\theta) \quad\text{as } r\to+\infty,\quad \theta \ne\pm\theta_0 \text{ fixed} .
\end{eqnarray*}
Thus, solutions corresponding to different $\mu$ are distinct. They are local variational solutions if $\mu_- = 0$, and very weak solutions if $\mu_- > 0$
{and the sector is reentrant ($\alpha<1$). In a flat or salient sector ($\alpha\ge1$) the solutions corresponding to $\mu_- > 0$ are not $L^2$ in a neighbourhood of the origin, thus they do not qualify as very weak solutions.}

\begin{rem}
One may wonder if there exists a solution to \eqref{pde} that is not non-negative. Using the
maximum principle in unbounded domains of $\mathbb{R}^2$ (see \cite{BCN97}) we can prove
that any solution of \eqref{pde} that is bounded from below is non-negative.
\end{rem}

\subsection{Local regularity near the tip of the sector}
Using the tools of Proposition~\ref{local-rgty} (localization, bootstrapping, and regularity theory for the linear Poisson--Dirichlet problem), one obtains the following results.
\begin{pro}
\label{local-rgty-gene}
There holds:
\begin{enumerate}
\item Any solution to~\eqref{pde} belongs to $C^\infty(\mathcal{K})$, for any compact subset $\mathcal{K} \subset \overline\Omega$ such that the origin $0 \notin \mathcal{K}$.
\item If $\Omega$ is salient of flat ($\theta_0 \le \pi/2$), any very weak solution $\varphi$ is actually a local variational solution, and all solutions satisfy $\varphi \in W^{2,p}(\Omega_R) \subset C^1(\overline{\Omega_R})$ for some $p>2$ and all finite~$R$.
\item It $\Omega$ is reentrant ($\theta_0 > \pi/2$), then:
\begin{itemize}
\item any local variational solution satisfies
$\varphi \in H^s(\Omega_R) \subset C(\overline{\Omega_R})$, for all $s < 1 + \alpha$, and admits the expansion $\varphi(r,\theta) = \lambda\, S(r,\theta) + \tilde \varphi(r,\theta)$, with $\tilde \varphi \in H^2(\Omega_R)$, and $\lambda \ge \Lambda$
{if $\varphi\ge0$};
\item any very weak solution admits the expansions
\begin{eqnarray}
 \varphi(r,\theta) &=& \lambda_*\,S^*(r,\theta) + \hat\varphi(r,\theta),\quad \hat\varphi \in H^1(\Omega_R)\,;
\\
 &=& \lambda_*\,S^*(r,\theta) + \lambda\, S(r,\theta) + \tilde \varphi(r,\theta),\quad \tilde\varphi \in H^2(\Omega_R).\qquad
\end{eqnarray}
for all finite $R$. The coefficient $\lambda_*$ is non-negative.
\end{itemize}
\end{enumerate}
\end{pro}
\begin{proof}
We only prove the last claim; the others are similar. By definition, a variational solution belongs to $ H^1(\Omega_R)$ for all finite~$R$; computing as in~\eqref{eq:chi.u} we see that $\chi_B\,\varphi \in \Phi_2(\Omega_{2B})$, hence the decomposition. As {any solution $\varphi\ge0$} is larger than the minimal solution~$u$, its singularity coefficient $\lambda$ is larger than that of~$u$, namely~$\Lambda$.

\medbreak

If $\varphi$ only is a very weak solution, one finds $\Delta(\chi_B\,\varphi) \in H^{-1}(\Omega_{2B})$. Introducing $v\in H^1_0(\Omega_{2B})$ such that $\Delta v = \Delta(\chi_B\,{\varphi})$ and $P := \chi_B\,\varphi -v$, one finds $P \in N(\Omega_{2B})$; thus, there is $\lambda_*\in\R$ such that $P = \lambda_*\, P_s^{2B} = \lambda_*\, (\chi_B\,S^* + \tilde P^{2B})$ and $\chi_B\,\varphi = \lambda_*\, \chi_B\,S^* + \varphi_B$, with $\varphi_B \in H^1_0(\Omega_{2B})$. As $B$ is arbitrary, we deduce that for any~$R$
$$\varphi(r,\theta) = \lambda_*\,S^*(r,\theta) + \hat\varphi(r,\theta),\quad \hat\varphi \in H^1(\Omega_R). $$
As $\Delta\hat\varphi = \Delta\varphi$, one finds $\chi_B\,\hat\varphi \in \Phi_2(\Omega_{2B})$, hence the second decomposition of~$\varphi$. The condition $\lambda_*\ge0$ is implied by the assumption $\E^{-\varphi} \in H^{-1}(\Omega_R)$.
\end{proof}

\subsection{Unboundedness}
We know from Remark~\ref{unbounded} that the minimal solution in a flat or reentrant sector is unbounded, hence all non-negative solutions are unbounded. Actually, the same holds in a salient sector.

\begin{pro}
Let $\Omega[\theta_0]$ be a salient sector: $\theta_0 < \pi/2$. Any non-negative solution to~\eqref{pde} satisfies:
\begin{equation}
\varphi(r,\theta) \ge \log(1 + \tfrac18\, r^2\,\sin^2(\theta_0 - |\theta|)).
\label{eq:unbound:salient}
\end{equation}
\end{pro}
\begin{proof}
Let $\xvec_0(r,\theta) \in \Omega$. The radius of the largest disk $D$ centred at $\xvec_0$ and contained in~$\Omega$ is $R = r\,\sin(\theta_0 - |\theta|)$. We have $\varphi \in H^1(D)$ by Proposition~\ref{local-rgty-gene}. Introduce the variational solution $\ul\varphi$ to the problem
\begin{equation}
-\Delta \ul\varphi = {\E^{-\ul\varphi} } \text{ in } D,\quad \ul\varphi = 0.
\label{pde:disk}
\end{equation}
This solution can be computed explicitly:
\begin{equation*}
{\ul\varphi}(\xvec) = \log\left( \frac{(A^2 - |\xvec - \xvec_0|^2)^2}{8\,A^2} \right), \quad\text{with:}\quad A = \sqrt 2+ \sqrt{2+R^2}.
\label{eq:phi.A}
\end{equation*}
On the other hand, $\varphi$ is a supersolution to~\eqref{pde:disk}, and $\varphi \ge \ul\varphi$ on~$D$. In particular
\begin{equation*}
\varphi(\xvec_0) \ge \ul\varphi(\xvec_0) = \log\left( R^2 + 4 + \sqrt{8R^2 + 16} \right) - \log 8 \ge \log( 1 + R^2/8),
\end{equation*}
which is~\eqref{eq:unbound:salient}.
\end{proof}
\begin{rem}\label{unbounded-bis}
The same holds in a reentrant sector, with $R = r\, \sin( \min(\theta_0 - |\theta|, \frac\pi2) )$.
\end{rem}

\section{ Application to plasma equilibria}

In \cite{KaLP13} the authors study the properties of the solution $\phi$
to the problem
\begin{equation}\label{vlasovpoisson0}
- \Delta\phi=  \kappa\, \exp(\phie-\phi):= \rho \text{ in } \Upsilon_1,\quad\text{with:}\quad
  \int_{\Upsilon_1} \rho =M
\end{equation}
in a bounded polygonal (or curvilinear polygonal) domain $\Upsilon_1$ of $\mathbb{R}^2$.
The \emph{mass} parameter~$M$ is a data of the problem; the normalization factor~$\kappa$ is an unknown which ensures that the constraint $\int_{\Upsilon_1} \rho =M$ is satisfied.
The external potential $\phie$ is fixed; it belongs to $L^\infty(\Upsilon_1)$.

\medbreak

{We shall need} the following assumptions.
\begin{enumerate}
\item Eq.~\eqref{vlasovpoisson0} is supplemented with a homogeneous Dirichlet boundary condition.
\item\label{global-shape} {If the domain $\Upsilon_1$ has reentrant corners, it is contained in its local tangent cone at any of them.}
\end{enumerate}
{The second condition is automatically satisfied when $\Upsilon_1$ is a straight polygon with only one reentrant corner: then, it}
can be described as $\Upsilon_1 = \Upsilon_1^C \setminus (\R^2 \setminus \Omega) \subset \Omega$, where $\Upsilon_1^C$ is the convex envelope of~$\Upsilon_1$, {and $\Omega$ is the tangent cone at the reentrant corner.}

\medbreak

We say that the domain $\Upsilon_1$ is \emph{regular} if it has no reentrant corner, i.e., it is smooth where it is not convex and convex where it is not smooth. Otherwise, it is \emph{singular}. To simplify the exposition, we shall only consider singular domains with one reentrant corner of opening $2\theta_0 = \pi/\alpha,\ 1/2 < \alpha <1$, and such that the two sides that meet there are locally straight. In that case, the decomposition of $\phi$ with respect to regularity writes  as in~\S\ref{ssc:rgty}:
\begin{equation}
\phi= \tilde\phi + \lambda\, \chi(r)\, r^{\alpha}\, \cos(\alpha \theta),\quad \tilde\phi \in H^2 \cap H^1_0(\Upsilon_1),
\label{dec:phi}
\end{equation}
assuming that the reentrant corner is located at~$0$ and the axes are suitably chosen.
Notice, however, that these conditions are not essential as long as assumption~\eqref{global-shape} above is satisfied. If the sides are not locally straight, the expression of the singular part is modified; the case of multiple reentrant corners is treated by localization.

\medbreak

We shall generally discuss the cases of regular and singular domains together; statements about singularity coefficients are void for a regular domain.
The goal of this section is to study the behaviour of the coefficients $\kappa$ and $\lambda$ as $M \to +\infty$. In~\cite{KaLP13} it was proved that:
\begin{equation}
\frac{M}{\kappa} \to 0 \quad \mbox{and}\quad \frac{\lambda}{\kappa} \to 0  \quad \mbox{and}\quad \lambda \to +\infty .
\label{eq:l6}
\end{equation}
We want to refine the estimates~\eqref{eq:l6} by obtaining explicit growth rates.

\subsection{Getting rid of the external potential}\label{ssc:phie}
For a fixed external potential $\phie$, the parameters $M$ and $\kappa$ are strictly increasing functions of each other~\cite{KaLP13}. Thus, the problem~\eqref{vlasovpoisson0} can be parametrized by~$\kappa$, even though $M$ is the significant variable. For fixed $\kappa$ and $\phie$, we provisionally denote $M[\kappa;\phie] = \int_{\Upsilon_1} \kappa\,\exp(\phie-\phi)$, where $\phi$ solves~\eqref{vlasovpoisson0}, and $\lambda[\kappa;\phie]$ the singularity coefficient of~$\phi$.

\smallbreak

Let $\phie^{\min},\ \phie^{\max}$ be the lower and upper bounds of~$\phie$. Obviously, for a given~$\kappa$ there holds:
\begin{equation}
M[\kappa\,\E^{\phie^{\min}};0] = M[\kappa;\phie^{\min}] \le M[\kappa;\phie] \le M[\kappa;\phie^{\max}] = M[\kappa\,\E^{\phie^{\max}};0] .
\label{e7.1}
\end{equation}

\medbreak

Similarly, we know~\cite{KaLP13} that for two given external potentials $(\phie^1,\phie^2)$ and normalization factors $(\kappa^1,\kappa^2)$ the corresponding solutions $(\phi^1,\phi^2)$ and their singularity coefficients satisfy
\begin{equation*}
\phie^1 + \log\kappa^1 \ge \phie^2 + \log\kappa^2 \text{ in } {\Upsilon_1}
\ \Rightarrow\ \phi^1 \ge \phi^2  \text{ in } {\Upsilon_1} \ \Rightarrow\ \lambda^1 \ge \lambda^2 .
\end{equation*}
So we deduce:
\begin{equation}
\lambda[\kappa\,\E^{\phie^{\min}};0] = \lambda[\kappa;\phie^{\min}] \le \lambda[\kappa;\phie] \le \lambda[\kappa;\phie^{\max}] = \lambda[\kappa\,\E^{\phie^{\max}};0] .
\label{e7.2}
\end{equation}

\medbreak

Therefore, we shall mostly concentrate on the case $\phie=0$. Using the physical scaling $\kappa =\epsilon^{-2}$, with $\epsilon\to0$, we shall denote by $\phi_\epsilon$ the solution to~\eqref{vlasovpoisson0} with $\phie=0$, i.e.:
\begin{equation}\label{vlasovpoisson}
-\Delta \phi_\epsilon= \epsilon^{-2} {\E^{-\phi_\epsilon} } \text{ in } \Upsilon_1,\quad
\phi_\epsilon=0 \text{ on } \partial\Upsilon_1 \,;
\end{equation}
the corresponding mass will be written~$M_\epsilon$; and $\lambda_\epsilon$ is the singularity coefficient of~$\phi_\epsilon$.

\subsection{Limit of $\phi_\epsilon$ and $\lambda_\epsilon$}
{Now we assume that the domain~$\Upsilon_1$ is singular, and we}
set the origin at the reentrant corner. First, we prove by some blow-up argument:

\begin{pro}
For any $\xivec$ in the unbounded sectorial domain $\Omega$, the rescaled function $v_\epsilon(\xivec)=\phi_\epsilon(\epsilon \xivec)$, defined for $\epsilon$ small enough,
converges towards $u(\xivec)$, where $u$ is the minimal solution to~\eqref{pde}
in~$\Omega$.
\end{pro}

\begin{proof}
{Let $\Gamma_1$ be} the union of the two sides that meet at the reentrant corner~$0$, {and} $R$ the radius of the circle centered at~$0$ that is tangent {to} $\partial\Upsilon_1 \setminus \Gamma_1$; in other words we consider $R$ to be the maximum of $r$ such that $\Omega_r=\Omega[\theta_0]\cap \{|\xvec|<r\}$  is included in $\Upsilon_1$.

\medbreak

We now perform the blow-up argument by a dilation of factor $\frac 1 \epsilon$. We obtain that $v_\epsilon$ is the solution to the following problem in the dilated domain~$\Upsilon_{\frac 1 \epsilon}$:
\begin{equation}\label{dilated}
-\Delta v_\epsilon = {\E^{-v_\epsilon} } \text{ in } \Upsilon_{\frac 1 \epsilon},\quad
v_\epsilon=0 \text{ on } \partial\Upsilon_{\frac 1 \epsilon} \,,
\end{equation}
while $u_{[\theta_0]}$ is a supersolution {--- by assumption~\eqref{global-shape}, the dilated domain is contained in~$\Omega[\theta_0]$}. Furthermore, we have $\Omega_{\frac R \epsilon} \subset \Upsilon_{\frac 1 \epsilon}$; therefore, $v_\epsilon$ is a supersolution to the same problem set in $\Omega_{\frac R \epsilon}$.
Summarizing, we have for any $\xivec$ in~$\Omega_{\frac R \epsilon}$:
\begin{equation}\label{gendarme}
u_{\frac R \epsilon}(\xivec)\leq \phi_\epsilon(\epsilon \xivec) = v_\epsilon(\xivec) \leq u_{[\theta_0]}(\xivec),
\end{equation}
and the result follows promptly from the results of Subsection~\ref{limitingproblem}.
\end{proof}

\medbreak

Denoting the polar coordinates as $(r,\theta)$ for~$\xvec$, $(\rho = \epsilon^{-1}\,r,\theta)$ for~$\xivec$, the regular-singular decomposition of~$v_\epsilon$ writes:
\begin{equation}
v_\epsilon(\rho,\theta) = \Lambda_\epsilon\, \rho^{\alpha}\, \cos(\alpha\theta) + \tilde u_\epsilon(\rho,\theta),\quad \tilde u_\epsilon \in H^2(\Upsilon_{\frac 1 \epsilon}) \,;
\end{equation}
obviously, the singularity coefficient $\Lambda_\epsilon$ is related to that of~$\phi_\epsilon$ as: $\Lambda_\epsilon = \epsilon^\alpha\,\lambda_\epsilon$.

\medbreak

On the other hand, Eq.~\eqref{gendarme} implies that $\Lambda_\epsilon$ is bounded between the singularity coefficients of $u_{\frac R \epsilon}$ and~$u$. By Remark~\ref{rem:LambdaR}, one deduces $\Lambda_\epsilon \to \Lambda$. In other words, when $\phie=0$, there holds:
\begin{equation}
\lambda_\epsilon \sim \Lambda\,\epsilon^{-\alpha} \quad \text{as } \epsilon\to0, \quad\text{i.e.,}\quad \lambda[\kappa;0] \sim \Lambda\,\kappa^{\alpha/2} \quad \text{as } \kappa\to+\infty.
\label{eq:lambda}
\end{equation}

\subsection{Limit of $M_\epsilon$}
We now state and prove a result that {complements} the numerical evidence in~\cite{KaLP13}.

\begin{pro}\label{pro:masse}
Let $|\partial \Upsilon_1|$ {be} the perimeter of $\Upsilon_1$.
When $\epsilon \rightarrow 0$, there holds:
\begin{equation}\label{equivalentdelamasse}
M_\epsilon \sim \sqrt 2 |\partial \Upsilon_1| \epsilon^{-1}.
\end{equation}
\end{pro}

\begin{proof}
Consider $\phi_\epsilon$ solution to \eqref{vlasovpoisson}. The main idea is to
\begingroup
 write $M_\epsilon = \int_{\Upsilon_1} -\Delta\phi_\epsilon = - \int_{\partial \Upsilon_1} \partial_n \phi_\epsilon$, and to derive an equivalent to $\partial_n \phi_\epsilon$ as $\epsilon\to0$.
Unfortunately, in a non-smooth domain this equivalent is not uniform along~$\partial\Upsilon_1$, and cannot be straightforwardly integrated on this boundary; thus some technicalities are needed to overcome the difficulty.
\endgroup

\medbreak

Introduce the function $w_\epsilon(\xvec) := -2\log \epsilon -\phi_\epsilon(\xvec)$, which solves:
\begin{equation}\label{bbus1}
\begin{split}
\Delta w_\epsilon= {\E^{w_\epsilon} } \text{ in } \Upsilon_1, \\
w_\epsilon= -2\log \epsilon \text{ on } \partial\Upsilon_1.
\end{split}
\end{equation}
We know (see \cite{KaLP13}) that $w_\epsilon$ converges to the boundary blow-up solution (or large solution) to $\Delta w = \E^w$ in $\Upsilon_1$.

\smallbreak

For almost every point $\xvec_0$ on the boundary of $\Upsilon_1$ (except the corners), we have the interior and exterior sphere condition: there is a (small) ball $B$ that is included in~$\Upsilon_1$ {(resp.~$\R^2 \setminus \overline{\Upsilon_1}$)} and tangent at $\xvec_0$.

\begin{figure}[ht]
\centerline{\includegraphics[width=12cm]{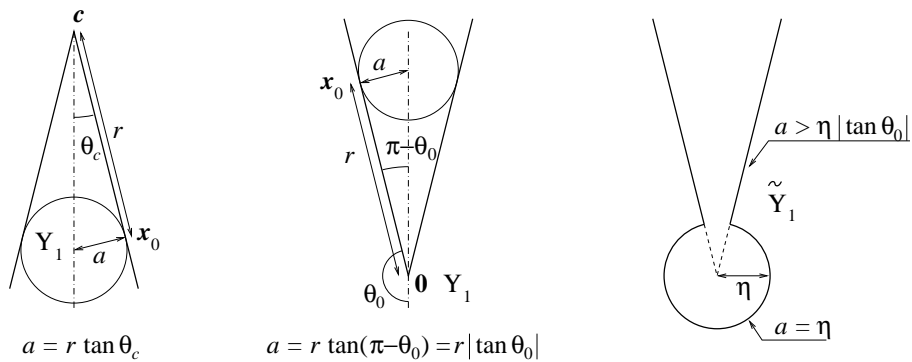}}
\caption{\emph{Left}: The interior sphere condition near a salient corner. \emph{Middle}: The exterior sphere condition near a reentrant corner. \emph{Right}: The mollified/rounded domain~$\tilde{\Upsilon}_1$.}
\label{fig:spheres}
\end{figure}

We begin with a lower bound for $M_\epsilon$. Fix $\eta>0$ small enough. We consider $\partial\Upsilon_{1,\eta}$ the set of $\xvec_0 \in \partial\Upsilon_{1}$ such that there exists a ball $B$ {\it of radius~$\eta$}, included in~$\Upsilon_1$ and tangent at $\xvec_0$.
\begingroup

If $\Upsilon_1$ is a straight polygon and $2\theta_c \in (0,\pi)$ is the angle at the corner~$\boldsymbol{c}$, the maximum radius of an interior sphere tangent at a nearby point $\xvec_0$ is $|\xvec_0 - \boldsymbol{c}| \, \tan \theta_c$; see Figure~\ref{fig:spheres} left. Elsewhere, this radius is bounded below by a constant. We deduce that there exists a constant~$K$ such that $|\partial\Upsilon_{1,\eta}| \ge |\partial\Upsilon_{1}|-K\eta$. This extends to a curvilinear polygon by diffeomorphism; anyway, $|\partial\Upsilon_{1}| - |\partial\Upsilon_{1,\eta}| \to0$ as $\eta\to0$.

\smallbreak

Fix $\xvec_0 \in \partial\Upsilon_{1,\eta}$, and let $B$ be the ball defined above. Solving the equation
\endgroup
\begin{equation}\label{bbus2}
\begin{split}
\Delta \overline{v}= {\E^{\overline{v}} }\; \mbox{in }\; B, \\
\overline{v} = -2\log \epsilon\; \mbox{on} \; \partial B,
\end{split}
\end{equation}
we easily find that $w_\epsilon$ is a subsolution to~\eqref{bbus2}, so $w_\epsilon(\xvec)\leq  \overline{v}(\xvec)$ for any $\xvec$ in~$B$, and then $\partial_n w_\epsilon (\xvec_0) \geq \partial_n \overline{v}(\xvec_0)$.

\medbreak

We now compute the normal derivative of the solution to $\Delta v = \E^v$ in a given ball. Up to a translation assume that the ball is centered at the origin.
{The solution is radially symmetric, i.e. we solve
\begin{equation}\label{radial1}
(r v'(r))'= r \E^{v(r)}.
\end{equation}
for $r<\eta$; and $\partial_n v = v'(\eta)\ge0$. Multiplying this equation by $r v'(r)$ and integrating between $0$ and $\eta$, we find:}
\begin{equation}\label{2mai1}
\frac{\eta^2 (v'(\eta))^2}{2}=\eta^2 \E^{-2\log \varepsilon} - {2\int_0^\eta r \E^{v(r)}\, dr} = \frac{\eta^2}{\epsilon^2} - 2\eta v' (\eta),
\end{equation}
\noindent using once more $r \E^{v(r)} = (r v'(r))'$ and the fact that $v=-2\log \epsilon$ {on} the boundary.
We infer from this equality:
\begin{equation}\label{2mai2}
v' (\eta)\geq \frac{2\eta\epsilon^{-2}}{\sqrt{2\eta^2\epsilon^{-2}}+2}.
\end{equation}
We then have
\begin{equation}\label{2mai3}
M_\epsilon\geq  \int_{\partial\Upsilon_{1,\eta}} \partial_n w_\epsilon (\xvec_0)\geq  \frac{2\eta\epsilon^{-2}}{\sqrt{2\eta^2\epsilon^{-2}}+2}\, (|\partial\Upsilon_{1}|-K\eta).
\end{equation}
Therefore
\begin{equation}\label{2mai4}
\liminf(\epsilon M_\epsilon)\geq \sqrt 2 (|\partial\Upsilon_{1}|-K\eta).
\end{equation}
We then let $\eta$ go to zero to obtain the lower bound.

\medbreak

We now proceed to the upper bound. {If $\xvec_0$ is not a corner,} we have another small ball $B_1$ included in $\mathbb{R}^2 \setminus \overline{\Upsilon_1}$ that is tangent at $\xvec_0$. Introduce another ball $B_2$, with the same center as~$B_1$, and large enough to have $\Upsilon_1 \subset B_2$.
Then consider the annulus $N=B_2 \setminus B_1$ that contains $\Upsilon_1$, and solve the boundary-value problem:
\begin{equation}\label{bbus3}
\begin{split}
&\Delta \underline{v}= {\E^{\underline{v}} } \text{ in } N, \\
& \underline{v} = -2\log \epsilon \text{ on } \partial N.
\end{split}
\end{equation}
The solution $\underline{v} \leq -2\log \epsilon$, thus it appears as a subsolution to~\eqref{bbus1}, and $\underline{v} \leq w_\epsilon$ in~$\Upsilon_1$;  at~$\xvec_0$, we have $\partial_n \underline{v} \geq \partial_n w_\epsilon$.

\medbreak

The rest of the proof amounts to compute the normal derivative of the solution to $\Delta v = \E^v$ in an annulus of radii $a<b$.
{Once again, the solution is radially symmetric, i.e. we solve~\eqref{radial1} for $a<r<b$. In the case of~\eqref{bbus3},}
we know that there exists $c\in (a,b)$ such that $v'(c)=0$, and that {$v'(a)\le0$}.
Multiplying~{\eqref{radial1}} by $r v'(r)$ and integrating between $r=a$ and~$c$, we have:
\begin{equation}\label{radial2}
- (a v'(a))^2 = 2\int_a^c {r^2 {\left(\E^{v(r)}\right)'}\, dr} = 2c^2\E^{v(c)}-2a^2 \E^{-2\log \epsilon}+4 a v'(a),
\end{equation}
{as in~\eqref{2mai1}. We}
thus obtain $|v'(a)|\leq \frac 2 a +\sqrt{{\frac 4{a^2}} + \frac {2}{ \epsilon^2}}\leq \frac 4 a + \frac {\sqrt 2}{ \epsilon }$.
This estimate provides us with a sharp {bound} of $\partial_n w$, but {it blows up near} the reentrant corner in the singular case: in this case $a \rightarrow 0$.
{{F}rom Figure~\ref{fig:spheres} middle, we see that the maximum radius of an exterior sphere tangent at $\xvec_0$ is $|\xvec_0 | / |\tan \theta_0|$, if the reentrant corner is located at~$0$.}

\medbreak

We overcome this difficulty as follows.
For $\eta>0$, consider a mollified/rounded domain $\tilde{\Upsilon}_1=\Upsilon_1 \setminus B(0,\eta)$, see Figure~\ref{fig:spheres} right. The maximum radius of the exterior sphere is equal to~$\eta$ on the rounded part of the boundary. As shown above, it is at least $\eta / |\tan \theta_0|$ on the remaining part of the sides that meet at the reentrant corner; elsewhere, it is bounded below by a constant.

\smallbreak

{Then introduce $\tilde{w}_\epsilon$} solution to
\begin{equation}\label{bbus18}
\begin{split}
\Delta \tilde{w}_\epsilon = {\E^{ \tilde{w}_\epsilon } }\; \mbox{in }\; \tilde{\Upsilon}_1, \\
\tilde{w}_\epsilon  = -2\log \epsilon\; \mbox{on} \; \partial\tilde{\Upsilon}_1,
\end{split}
\end{equation}
extended to $-2\log \epsilon$ {outside $\tilde{\Upsilon}_1$. As} $w_\epsilon$ is a subsolution to this problem, {there holds $w_\epsilon \le \tilde{w}_\epsilon$ on~$\tilde{\Upsilon}_1$, and also on~$\Upsilon_1 \setminus \tilde{\Upsilon}_1$.}
Setting $\tilde{M}_\epsilon= \int_{\tilde{\Upsilon}_1} {\E^{\tilde{w}_\epsilon }}$, we then have:
\begin{equation}\label{final18}
M_\epsilon \:= \int_{\Upsilon_1} {\E^{w_\epsilon}} \leq \tilde{M}_\epsilon +\int_{\Upsilon_1 \setminus \tilde{\Upsilon}_1} \frac{1}{\epsilon^2}.
\end{equation}
On the one hand, since each point of {$\partial\tilde{\Upsilon}_1$} satisfies the exterior sphere condition with a ball of radius {proportional to~$\eta$ (for $\eta$ small enough)}, we have:
\begin{equation}\label{final19}
\tilde{M}_\epsilon { \:=\:} \int_{\partial \tilde{\Upsilon}_1} \partial_n \tilde{w}_\epsilon \leq \left( \frac4{{c}\eta} + \frac {\sqrt 2}{ \epsilon } \right) \, |\partial \tilde{\Upsilon}_1|.
\end{equation}

\smallbreak

On the other hand $\int_{\Upsilon_1 \setminus \tilde{\Upsilon}_1} \frac{1}{\epsilon^2}\leq \frac{K\eta^2}{\epsilon^2}$
{and $|\partial \tilde{\Upsilon}_1| \le |\partial\Upsilon_{1}| + K'\eta$.}
Choosing $\eta=\epsilon^\frac23$ and gathering these estimates, we obtain
\begin{equation}\label{final20}
\epsilon M_\epsilon \leq \epsilon \left( \frac {4}{{c}\epsilon^\frac23} + \frac {\sqrt 2}{ \epsilon } \right) {\left( |\partial\Upsilon_{1}| + K'\epsilon^\frac23\right)} + K\epsilon^\frac13.
\end{equation}
Therefore it follows straightforwardly that $\limsup (\epsilon M_\epsilon)\leq \sqrt 2 \, |\partial\Upsilon_{1}|$.
This completes the proof of the proposition.
\end{proof}



\medbreak

\subsection{Conclusions}
{F}rom Proposition~\ref{pro:masse} we know that, when $\phie\equiv0$:
\begin{equation}
 M[\kappa;0] \sim \sqrt{2\kappa}\, |\partial\Upsilon_1| \quad \text{as } \kappa\to+\infty,
\label{eq:M}
\end{equation}
i.e., there are constants $C_1,\ C_2$ such that {(cf.~\eqref{eq:lambda}):}
\begin{equation}
\kappa \sim C_1\, M^2 \quad\text{and}\quad \lambda \sim C_2\, M^\alpha \quad \text{as } M\to+\infty.
\label{eq:kappalambda}
\end{equation}
When $\phie\ne0$, it follows from~\S\ref{ssc:phie} that
\begin{equation}
C_1\, M^2 \le \kappa \le C'_1\, M^2 \quad\text{and}\quad C_2\, M^\alpha \le \lambda \le C'_2\, M^\alpha \quad \text{as } M\to+\infty.
\label{eq:kappalambda:kg}
\end{equation}

\medbreak

\noindent{\bf Acknowledgements.} \quad
The authors thank L.~Dupaigne, A.~Farina and C.~Bandle for fruitful discussions about elliptic partial differential equations, {and} the anonymous referees for their useful remarks and suggestions.

\end{document}